\documentclass[reqno]{article}		
\usepackage{amssymb,amsmath,amsthm,enumitem}

\setlist{nolistsep}        

\newtheorem{lemma}{Lemma}
\newtheorem*{question}{Question}

\newcommand{\abs}[1]{\lvert#1\rvert}
\newcommand{\norm}[1]{\lVert#1\rVert}
\newcommand{\act}{\bullet}
\newcommand{\dom}{\mathsf{dom}}
\newcommand{\supp}{\mathsf{supp}}
\newcommand{\MIN}{\mathsf{MIN}}
\newcommand{\orbit}[1]{\mathsf{orb}(#1)}
\newcommand{\restr}{\mathord\lceil}
\newcommand{\psm}{\Delta}
\newcommand{\closure}[1]{\overline{#1}}
\newcommand{\Ub}{\mathsf{U_b}}
\newcommand{\BLipb}{\mathsf{BLip_b}}

\title{Another ambitable group}

\author{Jan Pachl}
\date{June 4, 2026}

\begin{document}

\maketitle

\begin{abstract}
It is an open question whether every topological group is precompact or ambitable.
This note presents a topological group that is ambitable but does not satisfy 
previously known sufficient conditions for being so.
\end{abstract}


\section{Group $G$}

The following question was asked in~\cite{Pachl2009atg}.

\begin{question}
Is every topological group either precompact or ambitable?
\end{question}

Theorem 3.35 in~\cite{Pachl2013usm} gave a positive answer
for a large class of topological groups.
That has left the question open for the remaining rather special groups.
The present note describes one such group $G$
to which the proof of the theorem does not apply.
Nevertheless, as is shown here using a somewhat different proof, $G$ is ambitable.

We work in the ZFC set theory.
As in~\cite{Kunen2011st}, every ordinal is the set of smaller ordinals,
$\omega=\omega_0$ is the least infinite ordinal, 
$\omega_{k+1}$ is the least ordinal of cardinality larger than $\omega_k$,
and $\omega_\omega = \bigcup_{k\in\omega} \omega_k$.
Hence 
$\omega\subseteq\omega_1\subseteq\omega_2\subseteq\dots\subseteq\omega_\omega$.

The set of elements of $G$ is the reduced product $\{0,1\}^{(\omega_\omega)}$; 
i.e., the set of the collections $x=\{x_\lambda\}_{\lambda\in\omega_\omega}$
of zeros and ones indexed by $\lambda\in\omega_\omega$
in which $x_\lambda =1$ for only finitely many $\lambda$.
The group operation is the pointwise addition,
$
x + y := \{ x_\lambda + y_\lambda \mod 2 \}_{\lambda\in\omega_\omega}
$.
The element $e_G =e\in G$ such that $e_\lambda = 0$ for all $\lambda$ 
is the identity of $G$.
Note that $-x=x$ and $x-y=x+y$ for $x,y\in G$.
Write
\begin{align*}
\supp(x) &:= \{ \lambda \in \omega_\omega \mid x_\lambda = 1 \}   
        \quad \text{for } x=\{x_\lambda\}_{\lambda\in\omega_\omega}\in G  \\
\MIN(x) &:= \min \{ k\in \omega \mid \supp(x) \cap \omega_k \neq \emptyset \}
                    \quad \text{for } x\in G, x\neq e_G  \\
\rho(x,y) &:= \frac{1}{1+\MIN(x+y)} \quad \text{for } x,y\in G, x\neq y 
\end{align*}
That defines an invariant metric $\rho$ which makes $G$ into a (metrizable abelian) 
topological group.

Basic notation and terminology here are as in~\cite{Pachl2013usm}. 
Functions and functionals are either real- or complex-valued;
the proof in the next section works in both cases.

\begin{itemize}
\item
For a bounded function $h$ on $G$, the sup norm of $h$ is
$\norm{h}=\sup_{x\in G} \abs{h(x)}$.
\item
The space of bounded uniformly continuous functions on $G$ is denoted $\Ub(G)$. 
\item
For $F\subseteq \Ub(G)$, $\closure{F}$ is the closure of $F$ 
in the space of all functions on $G$, in the pointwise topology.
\item 
For $h\in\Ub(G)$ and $y\in G$, define $y\act h \in \Ub(G)$
by $y\act h(x) := h(x+y)$, $x\in G$.
The orbit of $h$ is
$
\orbit{h} := \{ y\act h \mid y \in G \}
$.
\item
For a pseudometric $\psm$ on~$G$,
$\BLipb(\psm)$ is the set of functions $h$ on $G$ such that
$\norm{h}\leq 1$ and $\abs{h(x)-h(y)} \leq \psm(x,y)$ for all $x,y \in G$.
\end{itemize}
Note that for every pseudometric $\psm$ we have 
$\closure{\BLipb(\psm)} = \BLipb(\psm)$.
If moreover $\psm$ is continuous and invariant on $G$ then 
$\BLipb(\psm)\subseteq \Ub(G)$ and $\orbit{h} \subseteq \BLipb(\psm)$
for every $h\in \BLipb(\psm)$.

In the next section I prove that $G$ is ambitable.
That is, for every continuous invariant pseudometric $\psm$ on $G$
there exists $h\in\Ub(G)$ such that $\BLipb(\psm) \subseteq \closure{\orbit{h}}$.


\section{$G$ is ambitable}

To prove that $G$ is ambitable, 
take any continuous invariant pseudometric $\psm$ on $G$.
The proof will be carried out assuming that $\psm$ is a metric of the form
\begin{equation}
\psm(x,y) = s_{\MIN(x+y)} \quad\text{for } x,y\in G, x\neq y,
\label{cond:pmform}
\end{equation}
where $\{s_k\}_{k\in\omega}$ is a sequence such that
$s_0 > s_1 > \ldots > 0$ and $\lim_k s_k = 0$.
This assumption causes no loss of generality.
Indeed, given any continuous invariant pseudometric $\psm^\prime$ on $G$, we can find
(using for example the method in the proof of Lemma~3.3 in~\cite{Pachl2013usm})
a continuous invariant metric $\psm$ of the form~(\ref{cond:pmform})
and such that $\psm \geq \min(2,\psm^\prime)$,
and then $\BLipb(\psm^\prime)\subseteq \BLipb(\psm)$.

The key idea of the proof is to consider $G$ 
as the increasing union of subgroups $G_k$
defined by the increasing sequence of ordinals $\omega_k$.
For $k\in\omega$ define the subgroup
$G_k := \{ x\in G \mid \supp(x) \subseteq \omega_k \}$.
For $k\in\omega$,
denote by $\pi_k \colon G \to G_k$ the natural projection.
Thus $\pi_k (x) = y$ means that 
$y_\lambda = x_\lambda$ for $\lambda\in\omega_k$
and $y_\lambda = 0$ for $\lambda\not\in\omega_k$.
Note that $\pi_k$ is a continuous homomorphism from $G$ onto $G_k$.

We have 
$\psm(x,y) \leq s_{k+1}$ if and only if $\pi_k(x) = \pi_k(y)$.
In particular, $\psm(x,\pi_k(x)) \leq s_{k+1}$ for all $x\in G$.

The proof uses partial functions defined on subsets of $G$.
The domain of a partial function $\varphi$ is denoted $\dom(\varphi)$.
For a partial function $\varphi$ and a set $S$, 
${\varphi}\restr{S}$ denotes $\varphi$ 
restricted to the set $\dom(\varphi) \cap S$.
To simplify notation, write $\varphi/k$ instead of $\varphi\restr G_k$.

Take any countable dense subset $Q$ of the set of (real or complex)
numbers of absolute value $\leq 1$.
Let $\Phi$ be the set of all non-empty finite partial functions 
$\varphi$ from $G$ to $Q$ such that
\begin{align}
\varphi = f \restr \dom(\varphi) \quad&\text{for some } f \in 2\,\BLipb(\psm)
    \label{cond:lip}    \\
\pi_k(\dom(\varphi)) \subseteq \dom(\varphi) \quad&\text{for all } k\in\omega  
    \label{cond:prclosed}
\end{align}
For $k\in \omega$ write
$\Phi_k := \{ \varphi \in \Phi \mid \dom(\varphi) \subseteq G_k \}$.
We have $\abs{\Phi_k} = \abs{G_k} = \aleph_k$.

Note that $\Phi$ is ``dense'' in $\BLipb(\psm)$, in this sense: 
If $F\subseteq \Ub(G)$ is a set of functions such that 
\[
\forall \varphi\in\Phi \quad\exists f\in F \quad \varphi = f \restr \dom(\varphi)
\]
then $\BLipb(\psm) \subseteq \closure{F}$.
The proof will be complete once we find a function $h\in\Ub(G)$ 
for which the set $F=\orbit{h}$ has this property.

\begin{lemma}
    \label{cl:one}
There are $z_\varphi \in G$ for $\varphi \in \Phi$ such that
\begin{itemize}
\item[(i)]
if $\varphi\in\Phi_k$ then $z_\varphi \in G_k$
\item[(ii)]
$\pi_j(z_\varphi) = z_{\varphi/j}$ for every $j\in\omega$
\item[(iii)]
if $\varphi,\psi\in\Phi$, $\varphi \neq \psi$ then
$(\dom(\varphi) + z_\varphi) \cap (\dom(\psi) + z_\psi ) = \emptyset$
\end{itemize}
\end{lemma}

\noindent
Note that $\pi_j(z_\varphi) = z_{\varphi/j}$ in property (ii)
trivially holds for $j\geq k$ when $\varphi\in\Phi_k$ and $z_\varphi \in G_k$.

\begin{proof}
Proceed by two levels of recursion. The outer recursion is over $k\in\omega$.

\emph{Basis, $k=0$:}
Let $<$ be a well ordering of $\Phi_0$ isomorphic to $\omega$.
Proceed by recursion in $(\Phi_0,<)$.
When $\chi\in\Phi_0$ is such that $z_\varphi \in G_0$ have been chosen for 
$\varphi\in\Phi_0$, $\varphi<\chi$, the set
\[
A := \dom(\chi) + \bigcup \{ \dom(\varphi) + z_\varphi 
    \mid \varphi\in\Phi_0, \varphi < \chi \}
\]
is finite.
Hence there is $z_\chi \in G_0 \setminus A$.
Then $(\dom(\varphi) + z_\varphi) \cap (\dom(\chi) + z_\chi ) = \emptyset$
for every $\varphi < \chi$.

\emph{Recursive step:}
Fix $k\in\omega$ and assume there are $z_\varphi \in G_k$ satisfying
(i), (ii) and (iii) for all $\varphi,\psi \in \Phi_k$.
Write $\Psi := \Phi_{k+1} \setminus \Phi_{k}$
and let $<$ be a well ordering of $\Psi$ isomorphic to $\omega_{k+1}$.
Proceed by recursion in $(\Psi,<)$.
When $\chi\in\Psi$ is such that $z_\varphi \in G_{k+1}$ have been chosen for 
$\varphi\in\Psi$, $\varphi<\chi$,
the set
\[
A := \dom(\chi) + \left( G_{k} \cup \bigcup \{ \dom(\varphi) + z_\varphi 
    \mid \varphi\in\Psi, \varphi < \chi \} \right)
\]
has cardinality $\aleph_{k}$.
Since $\abs{G_{k+1} \cap \pi^{-1}_{k} (z_{\chi/k})} = \aleph_{k+1}$,
there exists 
$z_\chi \in (G_{k+1} \cap \pi^{-1}_{k} (z_{\chi/k})) \setminus A$.
It follows that
\begin{itemize}
\item
$\pi_{k}(z_\chi) = z_{\chi/k}$,
and therefore $\pi_{j}(z_\chi) = z_{\chi/j}$ for all $j\in\omega$;
\item
$(\dom(\varphi) + z_\varphi) \cap (\dom(\chi) + z_\chi ) = \emptyset$
for every $\varphi\in\Phi_{k}$;
\item
$(\dom(\varphi) + z_\varphi) \cap (\dom(\chi) + z_\chi ) = \emptyset$
for every $\varphi\in\Psi$, $\varphi < \chi$.
\qedhere
\end{itemize}
\end{proof}

\begin{lemma}
    \label{cl:two}
Let $z_\varphi \in G$ for $\varphi \in \Phi$ be as in Lemma~\ref{cl:one}.
Let $k\in\omega$, $\varphi,\psi\in\Phi$, $x\in\dom(\varphi)$ and $y\in\dom(\psi)$
be such that $\psm(x+z_\varphi,y+z_\psi) \leq s_{k+1}$.
Then $\abs{\varphi(x)-\psi(y)} \leq 4 s_{k+1}$.
\end{lemma}

\begin{proof}
As noted above,
$\psm(x+z_\varphi,y+z_\psi) \leq s_{k+1}$ is equivalent to
$\pi_k(x + z_\varphi) = \pi_k(y + z_\psi)$.
That together with (ii) in Lemma~\ref{cl:one} yields
\[
\pi_k(x) + z_{\varphi/k} = \pi_k(x) + \pi_k(z_\varphi) = \pi_k(x + z_\varphi)
= \pi_k(y + z_\psi) = \pi_k(y) + \pi_k(z_\psi) = \pi_k(y) + z_{\psi/k} \; .
\]

By (\ref{cond:prclosed}) we have 
$\pi_k(x)\in \dom(\varphi/k)$ and $\pi_k(y)\in \dom(\psi/k)$.
Hence $\varphi/k=\psi/k$ by (iii) in Lemma~\ref{cl:one},
$\pi_k(x) = \pi_k(y)$,
and $\varphi(\pi_k (x)) = \psi(\pi_k (y))$.
Since $\psm(x,\pi_k(x)) \leq s_{k+1}$ and $\psm(y,\pi_k(y)) \leq s_{k+1}$,
from~(\ref{cond:lip}) we get
\begin{align*}
\abs{\varphi(x) - \psi(y)} 
&\leq \abs{\varphi(x) - \varphi(\pi_k(x))} + \abs{\psi(\pi_k(y)) - \psi(y)} \\
&\leq 2\psm(x,\pi_k(x)) + 2\psm(y,\pi_k(y)) \leq 4 s_{k+1} .
\qedhere
\end{align*}
\end{proof}

With the two lemmas we are now ready to prove that $G$ is ambitable.

Take $z_\varphi \in G$ for $\varphi \in \Phi$ as in Lemma~\ref{cl:one}.
Let $D:= \bigcup \{ \dom(\varphi) + z_\varphi \mid \varphi \in \Phi \}$.
By Lemma~\ref{cl:two}, the following function $\tilde{h}\colon D \to Q$
is well defined and uniformly continuous on the subspace $D$ of $G$:
\[
\tilde{h}(x + z_\varphi ) := \varphi(x) \quad\text{for } x \in \dom(\varphi) .
\]
It follows~\cite[2.22]{Pachl2013usm} that $\tilde{h}$ extends to a function $h\in\Ub(G)$. 
We have $\varphi = ({z_\varphi}\act h )\restr \dom(\varphi)$
for every $\varphi\in\Phi$.
Hence $\BLipb(\psm) \subseteq \closure{\orbit{h}}$, q.e.d.

\vspace{8mm}

{\bf Acknowledgement.} 
I wish to thank Juris Stepr\={a}ns for discussions of and comments
on the content of this note.


\end{document}